\documentclass[a4paper, 12pt, american]{amsart}

\usepackage[colorlinks,pdfpagelabels,pdfstartview = FitH,bookmarksopen
= true,bookmarksnumbered = true,linkcolor = blue,plainpages =
false,hypertexnames = false,citecolor = red,pagebackref=false]{hyperref}
\usepackage{times,latexsym,amssymb}
\usepackage{amsmath,amsthm,bm}
\usepackage{graphics, epsfig}
\usepackage{color}
\usepackage{appendix}
\usepackage[margin=1.2in]{geometry}

\usepackage[T1]{fontenc}
\usepackage{amsbsy}
\usepackage{amstext}
\usepackage{amssymb}
\usepackage{esint}
\usepackage{stmaryrd}

\renewcommand{\d}{\mathrm{d}}
\newcommand{\dx}{\mathrm{d}x}
\newcommand{\dy}{\mathrm{d}y}

\newcommand{\dt}{\mathrm{d}t}
\renewcommand{\rho}{\varrho}

\newcommand{\essliminf}{\operatornamewithlimits{ess\,lim\,inf}}
\newcommand{\esslimsup}{\operatornamewithlimits{ess\,lim\,sup}}

\let\TeXchi\chi
\newbox\chibox
\setbox0 \hbox{\mathsurround0pt $\TeXchi$}
\setbox\chibox \hbox{\raise\dp0 \box 0 }
\def\chi{\copy\chibox}

\def\Xint#1{\mathchoice
    {\XXint\displaystyle\textstyle{#1}}%
    {\XXint\textstyle\scriptstyle{#1}}%
    {\XXint\scriptstyle\scriptscriptstyle{#1}}%
    {\XXint\scriptscriptstyle\scriptscriptstyle{#1}}%
    \!\int}
\def\XXint#1#2#3{\setbox0=\hbox{$#1{#2#3}{\int}$}
    \vcenter{\hbox{$#2#3$}}\kern-0.5\wd0}
\def\bint{\Xint-}
\def\dashint{\Xint{\raise4pt\hbox to7pt{\hrulefill}}}

\def\Xiint#1{\mathchoice
    {\XXiint\displaystyle\textstyle{#1}}%
    {\XXiint\textstyle\scriptstyle{#1}}%
    {\XXiint\scriptstyle\scriptscriptstyle{#1}}%
    {\XXiint\scriptscriptstyle\scriptscriptstyle{#1}}%
    \!\iint}
\def\XXiint#1#2#3{\setbox0=\hbox{$#1{#2#3}{\iint}$}
    \vcenter{\hbox{$#2#3$}}\kern-0.5\wd0}
\def\biint{\Xiint{-\!-}}
\author[N. Liao]{Naian Liao}
\address{Naian Liao,
Fachbereich Mathematik, Universit\"at Salzburg,
Hellbrunner Str. 34, 5020 Salzburg, Austria}
\email{naian.liao@sbg.ac.at}
\newtheorem{proposition}{Proposition}[section]
\newtheorem{theorem}{Theorem}[section]
\newtheorem{lemma}{Lemma}[section]

\newtheorem{remark}{Remark}[section]

\numberwithin{equation}{section}
\numberwithin{theorem}{section}
\numberwithin{proposition}{section}
\numberwithin{lemma}{section}
\numberwithin{remark}{section}
\newcommand{\noi}{\noindent}

\newcommand{\dsty}{\displaystyle}
\newcommand{\txty}{\textstyle}

\newcommand{\al}{\alpha}

\newcommand{\gm}{\gamma}
\newcommand{\dl}{\delta}
\newcommand{\Dl}{\Delta}

\newcommand{\varep}{\varepsilon}

\newcommand{\vp}{\varphi}
\newcommand{\sig}{\sigma}

\newcommand{\om}{\omega}
\newcommand{\Om}{\Omega}

\newcommand{\z}{\zeta}

\newcommand{\rr}{\mathbb{R}}
\newcommand{\rn}{\rr^N}

\newcommand{\bl}[1]{\mathbf{#1}}

\newcommand{\dvg}{\operatorname{div}}

\newcommand{\essup}{\operatornamewithlimits{ess\,sup}}
\newcommand{\essinf}{\operatornamewithlimits{ess\,inf}}

\newcommand{\loc}{\operatorname{loc}}

\newcommand{\pl}{\partial}

\newcommand{\intl}{\int\limits}

\def\Xint#1{\mathchoice
    {\XXint\displaystyle\textstyle{#1}}%
    {\XXint\textstyle\scriptstyle{#1}}%
    {\XXint\scriptstyle\scriptscriptstyle{#1}}%
    {\XXint\scriptscriptstyle\scriptscriptstyle{#1}}%
    \!\int}
\def\XXint#1#2#3{\setbox0=\hbox{$#1{#2#3}{\int}$}
    \vcenter{\hbox{$#2#3$}}\kern-0.5\wd0}
\def\bint{\Xint-}
\def\dashint{\Xint{\raise4pt\hbox to7pt{\hrulefill}}}
\def\dashiint{\bint\kern-0.15cm\bint}

\newcommand{\ovl}[3]{\int_{#1}^{#2}\kern-#3pt\raise4pt\hbox to7pt{\hrulefill}\ }

\newcommand{\ovll}[3]{\intl_{#1}^{#2}\kern-#3pt\raise4pt\hbox to7pt{\hrulefill}\ }

\newcommand{\tvl}[2]{\iint_{#1}\kern-#2pt\raise4pt\hbox to7pt{\hrulefill}\ }

\newcommand{\bye}{\end{document}}

\begin{document}
\title[Regularity of weak supersolutions]{Regularity of weak supersolutions to elliptic and parabolic equations:
lower semicontinuity and pointwise behavior}
\date{}
\maketitle
\begin{abstract}
We demonstrate a measure theoretical approach to the local regularity of weak supersolutions to elliptic and parabolic equations
in divergence form.
In the first part, we show that weak supersolutions become lower semicontinuous 
after redefinition on a set of measure zero.
The proof relies on a general principle, i.e. the De Giorgi type lemma,
which offers a unified approach for a wide class of elliptic and parabolic equations,
including an anisotropic elliptic equation, the parabolic $p$-Laplace equation,
and the porous medium equation.
In the second part, we shall show that for parabolic problems the lower semicontinuous representative at an instant
can be recovered pointwise from the  ``$\essliminf$'' of past times.
We also show that it can be recovered by the limit of certain integral averages of past times.
The proof hinges on the expansion of positivity for weak supersolutions.
Our results are structural properties of partial differential equations,
independent of any kind of comparison principle.

%

\end{abstract}
\section{Introduction}
The notion of {\it lower semicontinuity} plays a key role in potential theory.
One important example is in the notion of {\it superhamonic function} introduced by F. Riesz.
Let $\Om$ be an open subset of $\rr^d$ for some $d\ge1$.
A function $u:\Om\to (-\infty,+\infty]$ is called superharmonic if the following is satisfied:
\begin{enumerate}
 \item[(1)] $u$ is lower semicontinuous,
 \item[(2)] $u$ is finite in a dense subset of $\Om$,
 \item[(3)] for every ball $B\Subset \Om$ and every $h\in C(\bar{B})$ that is harmonic in $B$, 
  if $u\ge h$  on $\pl B$ then  $u\ge h$ in $B$.
\end{enumerate}
This is one way of interpreting the mnemonic inequality $\Dl u\le 0$ in $\Om$.
There is yet another way from calculus of variations involving Sobolev spaces.
Indeed, a function $u\in W^{1,2}_{\loc}(\Om)$ is called a weak supersolution to the Laplace equation $\Dl u=0$ if
for any compact set $K\subset\Om$
\[
\int_K Du\cdot D\vp\,\dx\ge0
\]
for all non-negative $\vp\in C^1_o(K)$. 
Note carefully that weak supersolutions are Sobolev functions and hence are merely defined almost everywhere,
whereas superharmonic functions are defined pointwise.

A natural question arises on the possible equivalence between the two notions, 
up to a set of measure zero.
While both directions are of interest, in this note we focus on the direction 
\begin{equation}\label{Eq:direction}
\text{ \{weak supersolutions\}} \quad\implies \quad\text{ \{superharmonic functions\}. }
\end{equation}
Incidentally, there is even a third notion of supersolution, i.e. the viscosity one, for which
lower semicontinuity is also a preset requirement.
For their equivalence and relevant references, \cite{Lindqvist} is a good source.

The notion of F. Riesz has also developed the nonlinear analogs and the parabolic counterparts
(cf. \cite{HKM, KL-96, KL-06, KL-08, KLL, KKP-10, Lindqvist, Watson}). Whereas the notion of weak supersolution is quite standard
for elliptic and parabolic equations in divergence form (cf. \cite{DB, DBGV-mono, GT, LU, LSU, Lieberman}).
A similar quest concerning the possible equivalence is certainly in order,
and thus the meaning of the implication in \eqref{Eq:direction} is enriched.
More specifically, we are interested in showing that {\it weak supersolutions
to a wide class of elliptic and parabolic equations possess lower semicontinuous representatives}.
Moreover, this turns out to be a {\it structural property} of partial differential equations in divergence form.
Together with the comparison principle under more stringent structural conditions of the partial differential equation,
which is certainly an issue of independent interest, this would allow us to conclude in this direction.

No doubt that the existence of lower semicontinuous representatives of weak supersolutions 
has been well-known for very general nonlinear elliptic equations such as the $p$-Laplace type $\Dl_p u=0$ 
(cf. \cite[\S~3.62]{HKM} and \cite{MZ, Trud}).
The main tool of the proof lies in a proper weak Harnack inequality for weak supersolutions.
This tool however is in general difficult to obtain, especially in the nonlinear parabolic setting.
Nevertheless, progress was made in \cite{Kuusi} where it was observed that a proper $L^r$ -- $L^{\infty}$ estimate
for weak {\it subsolutions} suffices to obtain a desired representative,
as long as the structure of the partial differential equation permits one to add a constant to
a solution to generate another one.
This observation dispenses with the demanding techniques in proving weak Harnack's inequality
and enables us to handle the parabolic $p$-Laplace type equation $u_t=\Dl_p u$,
and thus gives a new point of view even in the elliptic setting.
On the other hand, it seems not so handy in dealing with the porous medium type equation $u_t=\Dl (|u|^{m-1}u)$,
mainly because adding a constant to a solution does not guarantee another one.
When $m>1$, there was an attempt made for non-negative supersolutions in \cite{AL}.
See also \cite[Theorem~3.4]{BBGS} in this connection. Their approach is a certain adaption of the idea from \cite{Kuusi}.
However, it is unclear whether this approach can be applied in the setting of sign-changing solutions
to the porous medium type equation, let alone doubly nonlinear parabolic equations.
To give an answer is one goal of this note.

In the {\bf first part} of this note (cf. Section~\ref{S:1:1}), we demonstrate a new point of view: {\it the existence of
lower semicontinuous representative for weak supersolutions 
is in fact encoded in the De Giorgi type lemma}.
Loosely speaking, the lemma asserts that if $u$ is bounded below and
if the measure density of the set $\{x\in B_\rho(x_o):\,u(x)>\sig \text{ for some }\sig>0\}$ 
within a ball $B_\rho(x_o)$ of radius $\rho$ and center $x_o$ exceeds a critical number,
then we must have $u>\frac12\sig$ a.e. in $B_{\frac12\rho}(x_o)$. Geometrically, this means that
the minimum of $u$ must be attained in the annulus ``boundary" $B_{\rho}(x_o)\setminus B_{\frac12\rho}(x_o)$ and
the graph of $u$ will not allow any cusp; see Figure~\ref{Fig-1}.
\begin{figure}[t]\label{Fig-1}
\centering
\includegraphics[scale=.8]{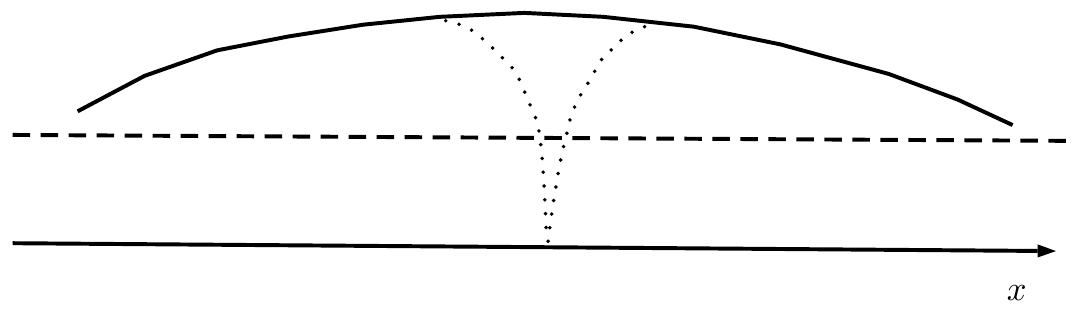}
\caption{No cusp}
\end{figure}
\noi It is obvious that
\[
u_*(x_o):=\lim_{\rho\to0}\essinf_{B_{\rho}(x_o)}u\le \lim_{\rho\to0}\bint_{B_{\rho}(x_o)}u\,\dy=u(x_o)\quad\text{ at every Lebesgue point }x_o.
\]
The lemma implies that the strict inequality cannot hold.
To illustrate the idea, let us suppose the strict inequality were to hold and $u_*(x_o)=0$ for simplicity.
Since $x_o$ is a Lebesgue point, by shrinking $\rho$,
the measure density of the set $\{x\in B_\rho(x_o):\,u(x)>\frac12u(x_o)\}$
would exceed the critical number. Consequently, we would have $u>\frac14 u(x_o)>0$ a.e. in $B_{\frac12\rho}(x_o)$ by the lemma,
which contradicts the definition of $u_*(x_o)=0$.
Therefore $u_*$ is a desired lower semicontinuous representative of $u$.

The gist of the De Giorgi type lemma lies in the ``from-measure-to-uniform" type estimate.
As such it can be derived 
from a proper $L^r$ -- $L^{\infty}$ estimate for the $p$-Laplace type equations (cf. \cite[Lemma~2.1]{Liao}).
Incidentally, the De Giorgi type lemma and the $L^r$ -- $L^{\infty}$ estimate are often referred to
as a local maximum principle in the literature.
However, their connection is not immediate for the porous medium equation.
Switching from the conventional $L^r$ -- $L^{\infty}$ estimate to the De Giorgi type lemma
offers considerable flexibility and economy, and allows us to treat 
signed supersolutions to the porous medium type equation for all $m>0$ simultaneously.
Hence, it seems that our approach reveals a more general and basic principle than the one in \cite{Kuusi}.
Since we do not impose any sign restriction on supersolutions, as a corollary,  weak {\it subsolutions}
will automatically enjoy {\it upper semicontinuous} representatives.
Therefore in contrast to the regularity theory of weak solutions to elliptic and parabolic equations, which reads

\

\hspace{.3cm}
\fbox{\begin{minipage}{12em}
$u$  is a locally bounded weak solution
\end{minipage}}\quad$\implies$\quad
\fbox{\begin{minipage}{12em}
 $u$ has a (H\"older) continuous representative,
\end{minipage}}

\

\noi our investigation can be illustrated as

\

\fbox{\begin{minipage}{13em}
$u$ is a weak super(sub)-solution locally bounded below(above)
\end{minipage}}\quad$\implies$\quad
\fbox{\begin{minipage}{13em}
 $u$ has a lower(upper) semicontinuous representative.
\end{minipage}}

\

\noi The local boundedness of super(sub)-solutions may or may not
be inherent in the preset notion of solution. 
This however deserves a separate investigation; see Remark~\ref{Rmk:2:2} and Remark~\ref{Rmk:2:8}.

To streamline the presentation, we will first formulate the so-called {\it property} \eqref{Property-D}
for measurable functions in Section~\ref{S:1:1}. This property \eqref{Property-D} is designed to recapitulate
the De Giorgi type lemmas satisfied by weak supersolutions to a variety of elliptic and parabolic equations
that interest us. Then we show that any locally integrable function with the property \eqref{Property-D}
has a lower semicontinuous representative; see Theorem~\ref{Thm:1:1}. To demonstrate the advantage of our approach,
we handle an anisotropic $p$-Laplace type elliptic equation in Section~\ref{S:anisotropic}
and a doubly nonlinear parabolic equation in Section~\ref{S:doubly}.
To our knowledge, the issue of lower semicontinuous representative for weak supersolutions to these
partial differential equations has never been dealt with before.

The {\bf second part} of this note (cf. Section~\ref{S:ptwise}) is devoted to the pointwise behavior of weak supersolutions
to parabolic equations.
This is a natural continuation of the study in the first part. Indeed, we have shown in the first part that the lower semicontinuous
representative $u_*$ coincides with $u$ at every Lebesgue point.
As is well known, Sobolev functions possess  Lebesgue points {\it almost everywhere} expect a set of ``capacity'' zero 
(cf.  \cite[Chapter 11]{DB-real-analysis} and \cite[Chapter 3]{Ziemer-89}). Thus it is natural to ask if $u_*$ can be recovered by 
the limit of the integral averages of $u$ at {\it every} point if $u$ is a supersolution. 
In the elliptic setting, this is well known (cf. \cite[\S~3.62]{HKM} and \cite{MZ}).
The parabolic case seems not well studied.
Due to the nature of parabolic equations,
one cannot expect to recover the point value of a supersolution at $(x,t)$ 
by the limit of the integral averages in standard Euclidean balls ``centered" at $(x,t)$.
A simple example would be
\begin{equation*}
H(x,t)=\left\{
\begin{array}{cc}
1,\quad t>0,\\
0,\quad t\le0.
\end{array}
\right.
\end{equation*}
One easily verifies that $H_*=H$ is a lower semicontinuous supersolution to the heat equation;
see Section~\ref{S:doubly} for the definition. 
Moreover, we have
\begin{itemize}
   \item $\dsty\lim_{\rho\to0}\bint_{-\rho}^{\rho}\bint_{B_{\rho}}|H-c|\,\dx\dt\neq0$\quad for any $c\in\rr$,
   \item $\dsty\lim_{\rho\to0}\bint_{-\rho}^{\rho}\bint_{B_{\rho}}H\,\dx\dt=\tfrac12\neq H_*(0,0)=0$.
\end{itemize}

In the potential theory for the heat equation, relying on explicit representations,
the recovery of a point value can be achieved by certain weighted integral averages over the so-called {\it heat balls},
which are attached to the designated instant from below; see \cite[Theorem~3.59]{Watson}.
Unlikely this will be a structural property we are seeking.
Nevertheless we will show in Proposition~\ref{Lm:Lebesgue-pt} for parabolic equations that
\begin{equation}\label{Eq:parabolic-Lebesgue}
\inf_{\theta>0}\lim_{\rho\to0}\bint_{t-2\theta\rho^p}^{t-\theta\rho^p}\bint_{B_{\rho}(x)}|u_*(x,t)-u(y,s)|\,\dy\d s=0\quad
\text{ for every }(x,t).
\end{equation}
Here $\rho^p$ reflects the natural time scaling of the particular parabolic equation.
Moreover, $\inf_{\theta>0}$ can be attained for non-degenerate equations.
This can be viewed as a certain {\it mean value property} of parabolic nature.
We have detected a similar statement as \eqref{Eq:parabolic-Lebesgue} in \cite[Theorem~2.1]{Ziemer-82},
in the context of the porous medium type equation ($m>1$). However, a sign restriction ($u_*<0$)
is imposed.
Like in the elliptic setting, the main tool used in \cite[Theorem~2.1]{Ziemer-82}
is a weak Harnack inequality for supersolutions. Thanks to the recent advances in the theory (cf. \cite{DBGV-acta, DBGV-mono}),
 we will demonstrate in Sections~\ref{S:3:1} -- \ref{S:3:3} that the {\it expansion of positivity} for supersolutions offers a handy
approach to \eqref{Eq:parabolic-Lebesgue}. The same approach can easily be adapted to the elliptic equations.
Thus, dispensing with the weak Harnack inequality we have offered a new perspective even in the elliptic setting. 
It is also worth mentioning that the expansion of positivity on the other hand lies
at the heart of any kind of Harnack's inequality (cf. \cite{DBGV-mono}).

Another result on the pointwise behavior of $u_*$ is the following important ``$\essliminf$" property
(cf. Theorem~\ref{Thm:pointwise}):
\begin{equation}\label{Eq:ptwise}
u_*(x,t)=\essliminf_{\substack{(y,s)\to(x,t)\\ s<t}}u(y,s)\quad\text{ for every }(x,t).
\end{equation}
Actually we will provide a finer version of \eqref{Eq:ptwise} in Theorem~\ref{Thm:pointwise}; see also Remark~\ref{Rmk:3:2}.
The equations  \eqref{Eq:parabolic-Lebesgue} -- \eqref{Eq:ptwise} reflect a distinctive feature of parabolic equations:
what is to happen in the future will have no influence at the present time.
A similar result as \eqref{Eq:ptwise} has been observed in \cite{KL-06, KL-08} in the context of 
potential theory for the prototype equations (for $p>2$ and $m>1$);
see \cite[Lemma~3.16, Corollary~3.53]{Watson} for the heat equation.
The significance of our result lies in that  \eqref{Eq:ptwise} holds true for a doubly nonlinear equation
which includes the parabolic $p$-Laplace type equation (for all $p>1$)
and the porous medium type equation (for all $m>0$) as special cases, and that
it is a structural property of weak supersolutions to parabolic equations in divergence form,
independent of any kind of comparison principle.
Needless to say, weak subsolutions will automatically be recovered as in \eqref{Eq:ptwise}
replacing ``$\essliminf$" by ``$\esslimsup$".
%

To summarize, the main goal of this note consists in demonstrating how the measure theoretical approach,
 which was originally invented by De Giorgi and has been developed recently
 in the quasilinear, degenerate parabolic setting (cf.  \cite{DBGV-acta, DBGV-mono}),
can be used to study the local regularity of weak supersolutions to a wide range of elliptic and parabolic equations,
some of which were inaccessible previously.
In my opinion, this approach offers considerable flexibility and more refined insight than the conventional one (cf. Table \ref{Table-1}).

\begin{table}[h]\label{Table-1}
\centering
\begin{tabular}{p{3.7cm}p{4.6cm}p{3.9cm} p{1cm}|p{1cm}|p{1cm}|p{1cm}|p{1.4cm}|}
\hline
Problems/Approaches&Conventional&Measure theoretical \\ \hline
Lower semicontinuity& $L^r$ -- $L^{\infty}$ estimate& De Giorgi type lemma\\ 
Pointwise behavior& Weak Harnack's inequality&Expansion of positivity\\
\hline
\end{tabular}
\caption{Summary}
\end{table}
{\it Acknowledgement.} This research has been funded by 
the FWF--Project P31956--N32 “Doubly nonlinear evolution equations”.
The author is grateful to Ugo Gianazza for carefully reading an early version
and valuable suggestions. 
\section{Lower semicontinuity}\label{S:1:1}
Let $\Om$ be an open set in $\rr^d$ for some integer $d\ge1$ and
$u$ be  a measurable function that is locally, essentially bounded below  in $\Om$.
For positive numbers $\rho$ and $\al_i$ with $i=1,2,\dots,d$, we define a cube in $\rr^d$ centered at the origin by
$$\mathcal{Q}_{\rho}=\prod_{i=1}^d(-\rho^{\al_i},\rho^{\al_i}).$$
If the center is shifted to some $y\in\rr^d$, then we write $\mathcal{Q}_{\rho}(y)$.
The positive parameters $\{\al_i:\,i=1,\dots,d\}$ will be chosen to reflect the scaling property of the 
particular partial differential equation.

Suppose $\mathcal{Q}_{\rho}(y)\subset \Om$ and introduce the real numbers $a$, $c$, $M$, and $\mu^-$ that satisfy
\begin{equation}\label{Eq:a-M-mu}
a, \,c\in(0,1),\quad M>0,\quad
\mu^-\le\essinf_{\mathcal{Q}_\rho(y)}u.
\end{equation}
We say $u$ satisfies the property \eqref{Property-D} in the sense that
\begin{equation}\tag{$\mathfrak{D}$}\label{Property-D}
\left.
	\begin{minipage}[c][2cm]{0.8\textwidth}
there exists a constant
$\nu\in(0,1)$ depending only on 
 $a$,  $M$, $\mu^-$ and other data, but independent of $\rho$, such that 
$$
[u\le \mu^- + M]\cap \mathcal{Q}_\varrho(y)|
	\le
	\nu| \mathcal{Q}_\varrho|
	\implies
u\ge \mu^- + a M \,\text{ a.e. in } \mathcal{Q}_{c \varrho}(y).
$$
	\end{minipage}
\right\}
\end{equation}
\begin{remark}\upshape
The constant $\nu$ could be a conglomerate of various quantities, such as  $a$,  $M$, $\mu^-$
and other data from the particular partial differential equation,
which will be specified in Section~\ref{S:anisotropic} and Section~\ref{S:doubly}.
The generosity of $\nu$'s dependence provides us with flexibility in adapting 
the property \eqref{Property-D} to different circumstances.
However the independence of $\rho$ is crucial.
\end{remark}

The {\it lower semicontinuous regularization} of $u$ is defined by
\begin{equation}\label{Eq:lsc-reg}
u_*(x):=\essliminf_{y\to x}u(y)=\lim_{r\to0}\essinf_{\mathcal{Q}_r(x)}u\quad\text{ for }x\in\Om.
\end{equation}
Note that because we have assumed $u$ is locally, essentially bounded below,
$u_*$  is well defined at every point of the domain.
It is not hard to see the so-defined $u_*$ is lower semicontinuous. 

Assume in addition that $u\in L^1_{\loc}(\Om)$ and denote the set of Lebesgue points of $u$ by 
\begin{equation}\label{Eq:Lebesgue-pt}
\mathcal{F}:=\left\{x\in \Om: |u(x)|<\infty,\quad \lim_{r\to0}\bint_{\mathcal{Q}_r(x)}|u(x)-u(y)|\,\dy=0\right\}.
\end{equation}
By the Lebesgue Differentiation Theorem \cite[Chapter~5, Section~11]{DB-real-analysis}, we have $|\mathcal{F}|=|\Om|$. 
\begin{theorem}\label{Thm:1:1}
Let the measurable function $u$ be locally integrable and locally, essentially bounded below in $\Om$.
Suppose that the property \eqref{Property-D} holds for $u$. Then $u(x)=u_*(x)$ for all $x\in\mathcal{F}$.
In particular, $u_*$ is a lower semicontinuous representative of $u$ in $\Om$.
\end{theorem}
\begin{proof}
It is obvious that $u_*(x)\le u(x)$ for all $x\in\mathcal{F}$ since
\[
u_*(x)=\lim_{r\to0}\essinf_{\mathcal{Q}_{r}(x)}u\le\lim_{r\to0}\bint_{\mathcal{Q}_r(x)}u\,\dy= u(x).
\]
 We will use the property \eqref{Property-D} to show the reverse inequality.
To this end, we pick $x_o\in\mathcal{F}$ and suppose to the contrary that
$u_*(x_o)<u(x_o)$.

Fix $R>0$ such that $\mathcal{Q}_R(x_o)\subset \Om$. Let  $\mu^-$ and $M$ satisfy 
$$\essinf_{\mathcal{Q}_R(x_o)}u=:\mu^-\le u_*(x_o)<\mu^-+M<u(x_o).$$
The quantities $u_*(x_o)$, $\mu^-$ and $M$ being fixed, 
we may choose $a\in(0,1)$ such that
\begin{equation}\label{Eq:choice-a}
\mu^-+aM>u_*(x_o),\quad\text{ i.e.}\quad \frac{u_*(x_o)-\mu^-}{M}<a<1.
\end{equation}
Such an $a$ being fixed in \eqref{Eq:choice-a}, we determine $\nu$
depending only on $a$, $\mu^-$, $M$ and other data but independent of $\rho$, according to the property \eqref{Property-D}. 

Next observe that there must be some $\rho\in(0,R)$ such that
\begin{equation}\label{Eq:critical-mass}
|[u<\mu^-+M]\cap \mathcal{Q}_\rho(x_o)|<\nu|\mathcal{Q}_\rho|.
\end{equation}
Otherwise, we would have arrived at 
\begin{align*}
\int_{\mathcal{Q}_\rho(x_o)}|u(x_o)-u(y)|\,\dy&\ge\int_{[u<\mu^-+M]\cap \mathcal{Q}_\rho(x_o)}\big[u(x_o)-(\mu^-+M)\big]\,\dy\\
&\ge\nu\big[u(x_o)-(\mu^-+M)\big]|\mathcal{Q}_\rho|
\end{align*}
for all $\rho\in(0,R)$, which is
a contradiction to the fact that $x_o\in\mathcal{F}$
since $\nu$ is independent of $\rho$.

By the property \eqref{Property-D}, the measure information \eqref{Eq:critical-mass} and the
choice of $a$ in \eqref{Eq:choice-a} imply that 
\[
u\ge\mu^-+aM>u_*(x_o)\quad\text{ a.e. in }\mathcal{Q}_{c\rho}(x_o).
\]
This however contradicts the definition of $u_*(x_o)$.
As a result, we must have $u_*(x)\ge u(x)$ for all $x\in\mathcal{F}$.
\end{proof}
\begin{remark}\label{Rmk:2:1}\upshape
The above proof mainly follows the argument already presented in the introduction.
However it is possible to give a more direct argument.
Indeed, for $\rho>0$ and $\mathcal{Q}_\rho(x_o)\subset\Om$,  let us introduce the non-negative quantity
\[
\om_{x_o}(\rho):=\bint_{\mathcal{Q}_\rho(x_o)}u\,\dy-\essinf_{\mathcal{Q}_\rho(x_o)}u.
\]
One could deem $\om_{x_o}$ as an analog of ``modulus of continuity" of $u$ at $x_o$. Suppose $x_o\in\mathcal{F}$ and set
\[
\varep_{x_o}(\rho):=\bint_{\mathcal{Q}_\rho(x_o)}[u(x_o)-u(y)]_+\,\dy.
\] 
A closer inspection of the above proof
reveals that if $u$ satisfies the property \eqref{Property-D}, then 
$\om_{x_o}$ can be quantified by $\varep_{x_o}$,
provided more explicit dependence of $\nu$ is known. 
In fact, for $\dl>0$ we may first estimate
\[
\dl\frac{|[u\le u(x_o)-\dl]\cap \mathcal{Q}_\rho(x_o)|}{|\mathcal{Q}_\rho|}\le\bint_{\mathcal{Q}_\rho(x_o)}[u(x_o)-u(y)]_+\,\dy=\varep_{x_o}(\rho).
\]
Suppose the constant $\nu$ in the property \eqref{Property-D} has
been fixed for the moment. We may choose $\dl\nu=\varep_{x_o}(\rho)$, such that
\[
\frac{|[u\le u(x_o)-\dl]\cap \mathcal{Q}_\rho(x_o)|}{|\mathcal{Q}_\rho|}\le\nu.
\]
Consequently, the property \eqref{Property-D} with 
the pairing $\mu^-+M=u(x_o)-\dl$ and $\mu^-+aM=u(x_o)-2\dl$ yields
\[
\essinf_{\mathcal{Q}_{\frac12\rho}(x_o)}u\ge u(x_o)-2\dl\quad\text{ with }\dl=(1-a)M,
\]
which in turn gives that
\begin{align*}
\om_{x_o}(\tfrac12\rho)&=\bint_{\mathcal{Q}_{\frac12\rho}(x_o)}u\,\dy-u(x_o)+u(x_o)-\essinf_{\mathcal{Q}_{\frac12\rho}(x_o)}u\\
&\le\varep_{x_o}(\tfrac{1}2\rho)+2\dl.
\end{align*}

On the other hand, the constant $\nu$ may depend on $\dl$ through $M$ (letting $a=\tfrac12$) 
according to the property \eqref{Property-D}.
Nevertheless, given amenable dependence of $\nu$ on $M$, for instance $\nu\approx M^\al$ with some $\al>0$, 
one could then represent $\dl$ by $\varep_{x_o}(\rho)$
through the relation $\dl\nu=\varep_{x_o}(\rho)$.
As a result, sending $\rho\to0$, the above estimate gives a more direct proof of Theorem~\ref{Thm:1:1}.
We shall see this is indeed the case when we consider particular partial differential equations
in the following Sections~\ref{S:anisotropic} -- \ref{S:doubly}.
\end{remark}
\subsection{An anisotropic $p$-Laplace equation}\label{S:anisotropic}
In this section, we work in an open set $E\subset\rr^N$ with $N\ge1$.
We consider the anisotropic $p$-Laplace equation
\begin{equation}\label{Eq:general}
\sum_{i=1}^{N}D_{x_i}A_i(x,u,Du)=0\quad\text{ weakly in }E,
\end{equation}
where the functions $A_i(x,u,\xi): E\times\rr\times\rr^{N}\to\rr$ are 
measurable with respect to $x \in E$ for all $(u,\xi)\in \rr\times\rn$
and continuous with respect to $(u,\xi)$ for a.e.~$x\in E$, 
and subject to the structure conditions
\begin{equation}\label{Eq:general-struc}
\begin{cases}
A_i(x,u,\xi)\cdot \xi_{i}\ge C_{o}|\xi_{i}|^{p_i},\\
|A_i(x,u,\xi)|\le C_{1}|\xi_{i}|^{p_i-1},
\end{cases}
\end{equation}
for some constants $p_i>1$, $C_{o}>0$ and $C_{1}>0$.
If all indices $p_i$'s are equal, \eqref{Eq:general} -- \eqref{Eq:general-struc} reduce to the standard $p$-Laplace
type equation.
In this section, we refer to the set of parameters $\{N,\,C_o,\,C_1,\,p_1,\dots,p_N\}$ as the {\it data}.
We will use $\gm$ as a generic positive constant that can be determined in terms of the data.

For a multi-index ${\bf p}=\{p_1,\dots,p_N\}$, let
\begin{equation*}
\left\{
\begin{array}{ll}
W^{1,{\bf p}}(E)=\{u\in L^1(E): u_{x_i}\in L^{p_i}(E),\, i=1,\dots,N\}\\[5pt]
W_o^{1,{\bf p}}(E)=W^{1,{\bf p}}(E)\cap W^{1,1}_o(E).
\end{array}\right.
\end{equation*}

A function $u\in W^{1,{\bf p}}_{\loc}(E)$ is called a local, weak supersolution to \eqref{Eq:general} -- \eqref{Eq:general-struc}
if for all compact set $K\subset E$
\[
\int_K \bl{A}(x,u, Du)\cdot D\vp\,\dx\ge0
\]
for all non-negative $\vp\in C^1_o(K)$. 
The notion of local, weak subsolution is defined by requiring $-u$ to be a local, weak supersolution.
A function that is both a weak supersolution and a weak subsolution is called a weak solution.

For $\rho>0$ we construct the cubes centered at the origin
\begin{equation*}
\mathcal{Q}_{\rho}(\theta)=\prod_{j=1}^N(-\rho_j,\rho_j),
\quad\rho_j=\theta\rho^{\frac{1}{p_j}}
\end{equation*}
for some parameter $\theta>0$. We denote the congruent cubes
centered at $y$ by $y+\mathcal{Q}_{\rho}(\theta)$.
Suppose $y+\mathcal{Q}_{\rho}(\theta)\subset E$ and $u$ is locally, essentially bounded below in $E$.
Let the numbers $a$, $\mu^-$ and $M$ be defined as in \eqref{Eq:a-M-mu}
with $\mathcal{Q}_\rho(y)$ replaced by $y+\mathcal{Q}_\rho(\theta)$.
The following De Giorgi type lemma can be retrieved from \cite{DBGV-16}.
\begin{lemma}\label{DG:anisotropic}
Let $u$ be a local, weak supersolution to \eqref{Eq:general} -- \eqref{Eq:general-struc} in $E$.
Assume that $u$ is locally, essentially bounded below in $E$.
There exists $\nu\in(0,1)$ depending only on $a$, $M$, $\theta$ and the data, such that
if
\[
|[u\le\mu^-+M]\cap [y+\mathcal{Q}_{\rho}(\theta)]|\le\nu|\mathcal{Q}_\rho(\theta)|,
\]
 then
\[
u\ge\mu^-+aM\quad\text{ a.e. in }y+\mathcal{Q}_{\frac12\rho}(\theta).
\]
Moreover, for some $c\in(0,1)$ depending only on the data, the constant $\nu$ has the form
\[
\nu=c (1-a)^N\bigg(\frac{M}{\theta}\bigg)^N.
\]
\end{lemma}

The parameter $\theta$ offers leeway in the method of {\it intrinsic scaling} 
from more advanced regularity theory (cf. \cite{DB, DBGV-mono,DBGV-16}),
which we will not evoke here. For our purpose, $\theta=1$ would be enough. Indeed,
Lemma~\ref{DG:anisotropic} implies that the property \eqref{Property-D} is satisfied by 
supersolutions to \eqref{Eq:general} -- \eqref{Eq:general-struc} with cubes $y+\mathcal{Q}_\rho\equiv y+\mathcal{Q}_\rho(1)$.
The lower semicontinuous regularization of $u$ and the set $\mathcal{F}$ of Lebesgue points
can be adapted from \eqref{Eq:lsc-reg} and \eqref{Eq:Lebesgue-pt} in Section~\ref{S:1:1} to the current setting after obvious changes. 
Then as a result of Lemma~\ref{DG:anisotropic} and Theorem~\ref{Thm:1:1}, we have the following.
\begin{theorem}
Let $u$ be a local, weak supersolution to \eqref{Eq:general} -- \eqref{Eq:general-struc} in $E$.
Assume that $u$ is locally, essentially bounded below in $E$.
Then $u_*(x)=u(x)$ for all $x\in\mathcal{F}$.
In particular, $u_*$ is a lower semicontinuous representative of $u$ in $E$.
\end{theorem}
\begin{remark}\label{Rmk:2:2}\upshape
The local boundedness of weak super(sub)-solutions to \eqref{Eq:general} -- \eqref{Eq:general-struc}
is an independent issue.
We refer to \cite{DBGV-16} for details and further references.
\end{remark}
\begin{remark}\upshape
Although it is not directly related to our concern in this note,
we mention that the H\"older regularity of locally bounded weak solutions to 
 \eqref{Eq:general} -- \eqref{Eq:general-struc} is still an open problem (cf. \cite{DBGV-16}).
\end{remark}
\subsection{A doubly nonlinear parabolic equation}\label{S:doubly}
Let $E$ be an open subset of $\rn$ and $E_T:=E\times(0,T)$ 
for some $T>0$.
In this section we consider the following doubly nonlinear parabolic equation
\begin{equation}  \label{Eq:p-q}
	\partial_t\big(|u|^{q-1}u\big)-\dvg\bl{A}(x,t,u, Du) = 0\quad \mbox{ weakly in $ E_T$}
\end{equation}
where $q>0$ and the function $\bl{A}(x,t,u,\xi)\colon E_T\times\rr^{N+1}\to\rn$ is only assumed to be
measurable with respect to $(x, t) \in E_T$ for all $(u,\xi)\in \rr\times\rn$,
continuous with respect to $(u,\xi)$ for a.e.~$(x,t)\in E_T$,
and subject to the structure conditions
\begin{equation}\label{Eq:1:2p}
	\left\{
	\begin{array}{c}
		\bl{A}(x,t,u,\xi)\cdot \xi\ge C_o|\xi|^p \\[5pt]
		|\bl{A}(x,t,u,\xi)|\le C_1|\xi|^{p-1}%
	\end{array}
	\right .
	\qquad \mbox{for a.e.~$(x,t)\in E_T$, $\forall\,(u,\xi)\in\rr\times\rn$,}
\end{equation}
where $C_o$ and $C_1$ are given positive constants, and $p>1$.
In this section, we will refer to the set of parameters $\{N,\,p,\,q,\,C_o,\,C_1\}$ as the {\it data}.
We will use $\gm$ as a generic positive constant that can be determined apriori in terms of the data.
\begin{remark}\label{Rmk:p-q}\upshape
In particular, when $q=1$ and $p>1$, \eqref{Eq:p-q} -- \eqref{Eq:1:2p} reduces to the parabolic $p$-Laplace type equation;
when $p=2$ and $q>0$, it is equivalent to the porous medium type equation.
To see the latter point, we only need to introduce a new function $v=|u|^{q-1}u$ and let $m=1/q$,
then \eqref{Eq:p-q} -- \eqref{Eq:1:2p} with $p=2$ represents the porous medium type equation about $v$.
\end{remark}
A function
\begin{equation*}  
	u\in L^{q}_{\loc}(E_T)\cap L^p_{\loc}\big(0,T; W^{1,p}_{\loc}(E)\big)
\end{equation*}
is a local, weak supersolution to \eqref{Eq:p-q} -- \eqref{Eq:1:2p}, if for every compact set $K\subset E$ and every sub-interval
$[t_1,t_2]\subset (0,T)$
\begin{equation*}  
	\int_{t_1}^{t_2}\int_{K} \left[-|u|^{q-1}u\z_t+\bl{A}(x,t,u,Du)\cdot D\z\right]\dx\dt
	\ge0
\end{equation*}
for all non-negative test functions $\z\in C_o^1\big(K\times (t_1,t_2)\big)$.
The notion of local, weak subsolution is defined by requiring $-u$ to be a local, weak supersolution.
A function that is both a weak supersolution and a weak subsolution is called a weak solution.
\begin{remark}\label{Rmk:notion-sol}\upshape
If the test function $\z$ is required to vanish only on the lateral boundary $\pl K\times(t_1,t_2)$,
we may replace $\z$ by $\vp_h(t)\z(x,t)$ for some small $h>0$, where $\vp_h$ is $1$ on $[t_1+h,t_2-h]$, vanishes outside $(t_1,t_2)$
and is linearly interpolated otherwise. Then after sending $h\to0$ we arrive at
\[
\int_K |u|^{q-1}u\z \,\dx\bigg|_{t_1}^{t_2}+\int_{t_1}^{t_2}\int_{K} \left[-|u|^{q-1}u\z_t+\bl{A}(x,t,u,Du)\cdot D\z\right]\dx\dt
	\ge0.
\]
Here the extra boundary terms are interpreted by the following limits:
\begin{align*}
&\int_{K\times\{t_1\}} |u|^{q-1}u\z \,\dx:=\lim_{h\to0}\bint_{t_1}^{t_1+h}\int_{K}|u|^{q-1}u\z\,\dx\dt,\\
&\int_{K\times\{t_2\}} |u|^{q-1}u\z \,\dx:=\lim_{h\to0}\bint_{t_2-h}^{t_2}\int_{K}|u|^{q-1}u\z\,\dx\dt.
\end{align*}
Such limits virtually hold for a.e. $t_1,\,t_2\in(0,T)$ in view of $u\in L^q_{\loc}(E_T)$.
Whenever it holds $t_1$ or $t_2$ will be called a {\it Lebesgue instant}.
\end{remark}

Let $K_\rho(y)\subset E$ be a cube centered at $y$,
with side length $2\rho$. Lest confusion arises,
we point out that $K_\rho(y)$ here is actually the same as $\mathcal{Q}_\rho(y)$ defined previously,
with all $\al_i=1$; whereas the letter $Q$ is reserved for  space-time cylinders in the parabolic setting.
When $y=0$ we simply write $K_\rho$.
For numbers $w$ and $k$, introduce the quantity
\begin{equation*}
	\mathfrak g(w,k):= q\int_{w}^{k}|s|^{q-1}(s-k)_-\,\d s\quad\text{ where }(s-k)_-:=\max\{k-s,0\}.
\end{equation*}
We first present the following energy estimate.
The proof is quite standard and can be given as in \cite[Proposition~3.1]{BDL}.
\begin{proposition}\label{Prop:2:1}
	Let $u$ be a  local, weak supersolution to \eqref{Eq:p-q} -- \eqref{Eq:1:2p} in $E_T$.
	There exists a constant $ \gm (C_o,C_1,p)>0$, such that
 	for all cylinders $Q_{R,S}:=K_R(y)\times (s-S,s)\Subset E_T$,
 	every $k\in\rr$, and every non-negative, piecewise smooth cutoff function
 	$\z$ vanishing on $\pl K_{R}(y)\times (s-S,s)$,  there holds
\begin{align*}
	\essup_{s-S<t<s}&\int_{K_R(y)\times\{t\}}	
	\z^p\mathfrak g (u,k)\,\dx
	+
	\iint_{Q_{R,S}}\z^p|D(u-k)_-|^p\,\dx\dt\\
	&\le
	\gm\iint_{Q_{R,S}}
		\Big[
		(u-k)^{p}_-|D\z|^p + \mathfrak g (u,k)|\z_t|\z^{p-1}
		\Big]
		\,\dx\dt\\
	&\phantom{\le\,}
	+\int_{K_R(y)\times \{s-S\}} \z^p \mathfrak g (u,k)\,\dx.
\end{align*}
\end{proposition}
\begin{remark}\label{Rmk:2:6}\upshape
The last space integral at the instant $s-S$ is interpreted as in Remark~\ref{Rmk:notion-sol}.
In particular, the above energy estimate implies improved integrability of $u$ immediately in the sense that
$u\in L^{\infty}_{\loc}\big(0,T; L^{q+1}_{\loc}(E)\big)$. Attention is called to that
 in general we do {\it not} have $u\in C_{\loc}\big(0,T; L^{q+1}_{\loc}(E)\big)$ for weak supersolutions, which holds true
for weak solutions (cf. \cite{DB, DBGV-mono, LSU, Lieberman}).
\end{remark}
For $(y,s)\in \rr^{N+1}$, we define the cylinders 
scaled by a positive parameter $\theta$:
\begin{equation}\label{cylinders}
\begin{cases}
\text{centered cylinders: }&(y,s)+Q_{\rho}(\theta)=
K_{\rho}(y)\times(s-\theta\rho^p, s+\theta\rho^p),\\
\text{forward cylinders: }&(y,s)+Q^+_{\rho}(\theta)=
K_{\rho}(y)\times(s, s+\theta\rho^p),\\
\text{backward cylinders: }&(y,s)+Q^-_{\rho}(\theta)=
K_{\rho}(y)\times(s-\theta\rho^p,s).
\end{cases}
\end{equation}
When $\theta=1$ or $(y,s)=(0,0)$ we omit them from the notation.
Let the numbers $a$, $\mu^-$ and $M$ be defined as in \eqref{Eq:a-M-mu}
with $\mathcal{Q}_\rho(y)$ replaced by $(y,s)+Q_{\rho}(\theta)$.
We use the above energy estimate to show the following De Giorgi type lemma.
\begin{lemma}\label{DG:p-q}
Let $u$ be a  local, weak supersolution to \eqref{Eq:p-q} -- \eqref{Eq:1:2p} in $E_T$.
Assume that $u$ is locally, essentially bounded below in $E_T$.
 There exists a constant
 $\nu\in(0,1)$ depending only on 
$a$, $M$, $\theta$, $\mu^-$ and the data, such that if
\[
|[u\le \mu^- + M]\cap [(y,s)+Q^-_\varrho(\theta)]|
	\le
	\nu | Q^-_\varrho(\theta)|
\]
then 
\[
u\ge \mu^- + a M \quad\text{ a.e. in } (y,s)+Q^-_{\frac34 \varrho}(\theta).
\]
\end{lemma}
\begin{proof}
Assume $(y,s)=(0,0)$. 
In order to employ the energy estimate in Proposition~\ref{Prop:2:1}, 
we notice first that 
for $\tilde k<k$ there holds $(u-k)_-\ge (u-\tilde k)_-$. Therefore, by the technical Lemma~2.2 in \cite{BDL},
 the energy estimate yields
 \begin{equation}\label{Eq:energy}
\begin{aligned}
	&\essup_{-\theta\varrho^p<t<0}
	\int_{K_\varrho}\z^p\big(|u|+|k|\big)^{q-1}(u-\tilde{k})_-^2\,\dx
	+
	\iint_{Q^-_\varrho(\theta)}\z^p|D(u-\tilde{k})_-|^p\,\dx\dt\\
	&\le
	\gm\iint_{Q^-_\varrho(\theta)}(u-k)^{p}_-|D\z|^p\,\dx\dt
	+
	\gm \iint_{Q^-_\varrho(\theta)}\big(|u|+|k|\big)^{q-1} (u-k)_-^2|\z_t|\,\dx\dt\\
	&=:I_1+I_2,
\end{aligned}
\end{equation}
for any non-negative piecewise smooth cutoff function $\zeta$ vanishing on the parabolic boundary of $Q^-_\varrho(\theta)$.
In order to use this energy estimate \eqref{Eq:energy}, we set
\begin{align*}
	\left\{
	\begin{array}{c}
\dsty k_n=\mu^-+aM+\frac{(1-a)M}{2^{n}},\quad \tilde{k}_n=\frac{k_n+k_{n+1}}2,\\[5pt]
\dsty \rho_n=\frac{3\rho}4+\frac{\rho}{2^{n+2}},\quad\tilde{\rho}_n=\frac{\rho_n+\rho_{n+1}}2,\\[5pt]
\dsty K_n=K_{\rho_n}, \quad \widetilde{K}_n=K_{\tilde{\rho}_n},\\[5pt]
\dsty Q_n=K_n\times(-\theta\rho_n^p,0),\quad\widetilde{Q}_n=\widetilde{K}_n\times(-\theta\tilde{\rho}_n^p,0).
\end{array}
	\right.
\end{align*}
Introduce the test function $\z$ vanishing on the parabolic boundary of $Q_{n}$ and
equal to the identity in $\widetilde{Q}_{n}$, such that
\begin{equation*}
|D\z|\le\gm\frac{2^n}{\rho}\quad\text{ and }\quad |\z_t|\le\gm\frac{2^{pn}}{\theta\rho^p}.
\end{equation*}
In this setting, we treat the two terms on the right-hand side of the energy estimate \eqref{Eq:energy} as follows.
For the first term, we estimate
\[
I_1\le\gm \frac{2^{pn}}{\varrho^p}M^{p}|A_n|,
\]
where
\begin{equation*}
	A_n=[u<k_n]\cap Q_n.
\end{equation*}
For the second term, \underline{when $q\ge1$}, we note that $\mu^-\le u\le k_n\le \mu^-+M$ on $A_n$ and estimate
\begin{align*}
I_2&\le\gm \frac{2^{pn}}{\theta\rho^p}
	\iint_{Q_n}\big(|u|+|k_n|\big)^{q-1} (u-k_n)^2_- \,\dx\dt \\
	&\le\gm \frac{2^{pn}}{\theta\rho^p} L^{q-1}M^{2}|A_n|,
\end{align*}
where $L:=\max\{| \mu^-|, |\mu^-+M|\}$.
\underline{When $0<q\le 1$}, we use the triangle inequality $|u|+|k_n|\ge(u-k_n)_-$ to estimate
\begin{align*}
I_2&\le\gm \frac{2^{pn}}{\theta\rho^p}
	\iint_{Q_n} (u-k_n)^{q+1}_- \,\dx\dt \\
	&\le\gm \frac{2^{pn}}{\theta\rho^p} M^{q+1}|A_n|.
\end{align*}
As a result, we may obtain that
\begin{align*}
	&\essup_{-\theta\tilde{\varrho}_n^p<t<0}
	\int_{\widetilde{K}_n} \big(|u|+|k_n|\big)^{q-1}(u-\tilde{k}_n)_-^2\,\dx
	+
	\iint_{\widetilde{Q}_n}|D(u-\tilde{k}_n)_-|^p \,\dx\dt\\
	&\qquad\le
	\gm \frac{2^{pn}}{\varrho^p}M^{p}\left(1+\frac{\max\{L^{q-1},M^{q-1}\} M^{2-p}}{\theta}\right)|A_n|,
\end{align*}
where $\gm$ depends on the data.

On the other hand, we recall $L=\max\{| \mu^-|, |\mu^-+M|\}$, so that $u\le\tilde k_n$ implies $|u|+|k_n|\le 2L$ and $|u|+|k_n|\ge k_n-u\ge k_n-\tilde k_n=(1-a)2^{-(n+2)}M$. Inserting this in the above energy estimate, we find that 
\begin{equation*}
\begin{aligned}
	\frac{(1-a)\min\{L^{q-1},M^{q-1}\}}{2^{q(n+3)}} &\essup_{-\theta\tilde{\varrho}_n^p<t<0}
	\int_{\widetilde{K}_n} (u-\tilde{k}_n)_-^2\,\dx
	+
	\iint_{\widetilde{Q}_n}|D(u-\tilde{k}_n)_-|^p \,\dx\dt\\
	&\le
	\gm \frac{2^{pn}}{\varrho^p}M^{p}\left(1+\frac{\max\{L^{q-1},M^{q-1}\}M^{2-p}}{\theta}\right)|A_n|.
\end{aligned}
\end{equation*}
Now setting $0\le\phi\le1$ to be a cutoff function which vanishes on the parabolic boundary of $\widetilde{Q}_n$
and equals the identity in $Q_{n+1}$, an application of the H\"older inequality  and the Sobolev imbedding
\cite[Chapter I, Proposition~3.1]{DB} gives that
\begin{align*}
	\frac{(1-a)M}{2^{n+3}}
	|A_{n+1}|
	&\le 
	\iint_{\widetilde{Q}_n}\big(u-\tilde{k}_n\big)_-\phi\,\dx\dt\\
	&\le
	\bigg[\iint_{\widetilde{Q}_n}\big[\big(u-\tilde{k}_n\big)_-\phi\big]^{p\frac{N+2}{N}}
	\,\dx\dt\bigg]^{\frac{N}{p(N+2)}}|A_n|^{1-\frac{N}{p(N+2)}}\\
	&\le\gm
	\bigg[\iint_{\widetilde{Q}_n}\big|D\big[(u-\tilde{k}_n)_-\phi\big]\big|^p\,
	\dx\dt\bigg]^{\frac{N}{p(N+2)}}\\
	&\quad\ 
	\times\bigg[\essup_{-\theta\tilde{\varrho}_n^p<t<0}
	\int_{\widetilde{K}_n}\big(u-\tilde{k}_n\big)^{2}_-\,\dx\bigg]^{\frac{1}{N+2}}
	 |A_n|^{1-\frac{N}{p(N+2)}}\\
	&\le 
	\gm 
	\bigg[\frac{2^{pn}}{\varrho^p}M^p\left(1+\frac{\max\{L^{q-1},M^{q-1}\}M^{2-p}}{\theta}\right)\bigg]^{\frac{N+p}{p(N+2)}}\\
	&\quad\times\bigg(\frac{(1-a)^{-1}2^{p(n+3)}}{\min\{L^{q-1},M^{q-1}\}}\bigg)^{\frac{1}{N+2}}
	|A_n|^{1+\frac{1}{N+2}}.
\end{align*}
In the last line we used the above energy estimate.
In terms of $ Y_n=|A_n|/|Q_n|$, this can be rewritten as
\begin{equation*}
\begin{aligned}
	 Y_{n+1}
	\le
	&\frac{\gm b^n}{(1-a)^{\frac{N+3}{N+2}}} \left(1+\frac{\max\{L^{q-1},M^{q-1}\}M^{2-p}}{\theta}\right)\\
	&\quad\times\left(\frac{\theta}{M^{2-p}\min\{L^{q-1},M^{q-1}\}}\right)^{\frac1{N+2}} Y_n^{1+\frac{1}{N+2}},
\end{aligned}
\end{equation*}
for positive constants $\gm$ and 
$b$ depending only on the data. 
Hence, by \cite[Chapter I, Lemma~4.1]{DB}, 
there exists
a positive constant $\nu$ of the form
\begin{align*}
\nu=&\frac{(1-a)^{N+3}}{\gm}\left(\frac{M^{2-p}\min\{L^{q-1},M^{q-1}\}}{\theta}\right)\\
&\quad\times\left(1+\frac{\max\{L^{q-1},M^{q-1}\}M^{2-p}}{\theta}\right)^{-(N+2)},
\end{align*}
 such that
$Y_n\to0$ if we require that $Y_o\le \nu$.
\end{proof}
\begin{remark}\upshape
When $q=1$, Lemma~\ref{DG:p-q} recovers Lemma~3.1 in Chapter~3 of \cite{DBGV-mono};
when $p=2$, it recovers simultaneously  Lemma~7.1 and  Lemma~10.1  in Chapter~3 of \cite{DBGV-mono}.
\end{remark}

Like in Section~\ref{S:anisotropic}, 
the parameter $\theta$ appears in the more advanced regularity theory of intrinsic scaling
(cf. \cite{BDL, DB, DBGV-mono}), which we will not evoke here.
For our purpose, $\theta=1$ is enough.
The only effort needed is to rephrase Lemma~\ref{DG:p-q} using centered cylinders.
Indeed, we still assume $(y,s)=(0,0)$ for simplicity and take $\theta=1$. Lemma~\ref{DG:p-q} implies
 there is $\nu\in(0,1)$ depending only on 
$a$, $M$, $\mu^-$ and the data, such that if
\[
|[u\le\mu^-+M]\cap Q_\varrho|
	\le
	\nu | Q_\varrho|
\]
then 
\[
u\ge\mu^-+a M \quad\text{ a.e. in } Q_{c\varrho}
\]
for some $c\in(0,1)$ depending only on $p$. See Figure~\ref{Fig-2}. 
\begin{figure}[t]\label{Fig-2}
\centering
\includegraphics[scale=1]{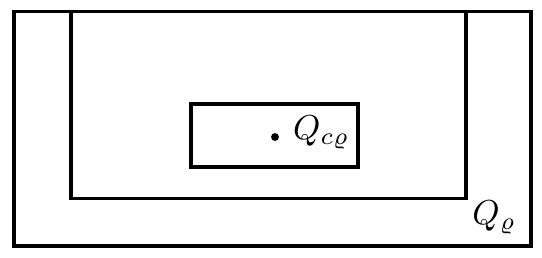}
\caption{Cylinders}
\end{figure}

The lower semicontinuous regularization of $u$ and the set $\mathcal{F}$ of Lebesgue points
can be adapted from \eqref{Eq:lsc-reg} and \eqref{Eq:Lebesgue-pt} in Section~\ref{S:1:1} to the current setting after obvious changes. 
Then as a result of Lemma~\ref{DG:p-q} and Theorem~\ref{Thm:1:1}, we have the following.
\begin{theorem}\label{Thm:lsu-parabolic}
Let $u$ be a local, weak supersolution to \eqref{Eq:p-q} -- \eqref{Eq:1:2p} in $E_T$.
Assume that $u$ is locally, essentially bounded below in $E_T$.
Then $u_*(x,t)=u(x,t)$ for all $(x,t)\in\mathcal{F}$.
In particular, $u_*$ is a lower semicontinuous representative of $u$ in $E_T$.
\end{theorem}
\begin{remark}\label{Rmk:2:8}\upshape
The local boundedness of weak super(sub)-solutions to  \eqref{Eq:p-q} -- \eqref{Eq:1:2p}
is an independent issue. We refer to \cite{DB, DBGV-mono} for this issue.
\end{remark}
\section{Pointwise behavior}\label{S:ptwise}
In this section, we study the pointwise behavior of the lower semicontinuous representative $u_*$
defined in Theorem~\ref{Thm:lsu-parabolic} for supersolutions to the parabolic equation \eqref{Eq:p-q} -- \eqref{Eq:1:2p}.
We keep using the cylinders defined in \eqref{cylinders}.
The first theorem asserts that for any point $(x,t)$ in the domain $E_T$,
the point value $u_*(x,t)$ can be recovered by the ``$\essliminf$" of previous times; see Figure~\ref{Fig:mean-value}.

\begin{theorem}\label{Thm:pointwise}
Let $u$ be a  local, weak supersolution to \eqref{Eq:p-q} -- \eqref{Eq:1:2p} in $E_T$,
Assume that $u$ is locally, essentially bounded below in $E_T$
and that $u_*$ is the lower semicontinuous representative of $u$ obtained in Theorem~\ref{Thm:lsu-parabolic}.
For every $(x,t)\in E_T$, there holds 
\begin{equation}\label{Eq:3:0}
u_*(x,t)=\inf_{\theta>0}\lim_{\rho\to0}\essinf_{(x,t)+Q'_{\rho}(\theta)}u,
\end{equation}
where
$(x,t)+Q'_{\rho}(\theta):=K_{\rho}(x)\times(t-2\theta\rho^p,t-\theta\rho^p)$. In particular, we have
\[
u_*(x,t)=\lim_{\rho\to0}\essinf_{(x,t)+Q^-_{\rho}}u\equiv\essliminf_{\substack{(y,s)\to(x,t)\\ s<t}}u(y,s)\quad\text{ for every }(x,t)\in E_T.
\]
\end{theorem}
\begin{figure}\label{Fig:mean-value}
\centering
\includegraphics[scale=1]{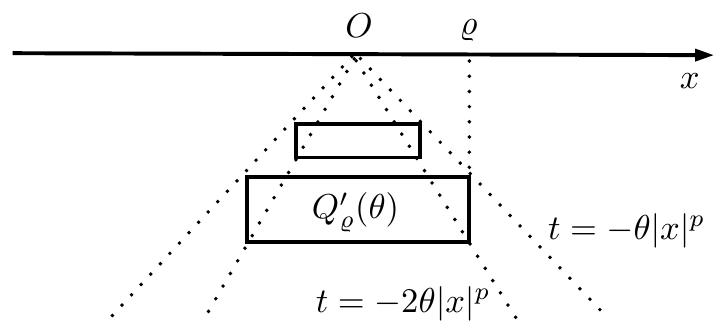}
\caption{Pointwise behavior}
\end{figure}
\begin{remark}\label{Rmk:3:1}\upshape
The infimum of $\theta$ in \eqref{Eq:3:0} can be attained at some real number, if $p=2$ and $q=1$; see Remark~\ref{Rmk:3:3}.
\end{remark}
\begin{remark}\label{Rmk:3:2}\upshape
More generally, the recovery of $u_*(x,t)$ via ``$\essliminf$"
could also take place along regions that are much smaller than $(x,t)+Q^-_{\rho}$ as $\rho\to0$.
For instance, assume $(x,t)=(0,0)$, let $\varep>0$ and 
define the region $\mathfrak{R}_{\rho}^\varep:=\{(x,t)\in\rn\times\rr: -|x|^{p+\varep}<t<0,\, |x|<\rho\}$.
Then we have
\[
u_*(0,0)=\lim_{\rho\to0}\essinf_{\mathfrak{R}_{\rho}^\varep} u.
\]
This seems not observed even when $p=2$. See \cite[Corollary~3.53]{Watson}
for the heat equation in this connection. 
\end{remark}
Now we prepare to start the proof.
By the definition of $u_*$ in \eqref{Eq:lsc-reg} (after an adaption), it is obvious that for every $(x,t)\in E_T$
there holds 
\begin{equation}\label{Eq:3:1}
u_*(x,t)=\essliminf_{(y,s)\to(x,t)}u(y,s)\le\inf_{\theta>0}\lim_{\rho\to0}\essinf_{(x,t)+Q'_{\rho}(\theta)}u.
\end{equation}
To show the reverse inequality, we need the following variant version of the De Giorgi type lemma
which involves certain ``initial data". 
Let the numbers $a$, $\mu^-$ and $M$ be defined as in \eqref{Eq:a-M-mu}
with $\mathcal{Q}_\rho(y)$ replaced by $(y,s)+Q_{\rho}(\theta)\subset E_T$. 
\begin{lemma}\label{DG:p-q-bdry}
Let $u$ be a  local, weak supersolution to \eqref{Eq:p-q} -- \eqref{Eq:1:2p} in $E_T$.
Assume that $u$ is locally, essentially bounded below in $E_T$.
 There exists a constant
$\theta>0$ depending only on 
$a$, $M$, $\mu^-$ and the data, such that if $s$ is a Lebesgue instant and
\[
u(\cdot, s)\ge \mu^- + M\quad\text{ a.e. in }K_\rho(y),
\]
then 
\[
u\ge \mu^- + a M \quad\text{ a.e. in } (y,s)+Q^+_{\frac34 \varrho}(\theta).
\]
\end{lemma}
\begin{proof}
Assume $(y,s)=(0,0)$. We intend to use the energy estimate of Proposition~\ref{Prop:2:1}
in $Q_{R,S}\equiv Q^+_{\rho}(\theta)$.
Note that the time level $s-S$ in Proposition~\ref{Prop:2:1} corresponds to $t=0$ here.
Let $\z(x)$ be a time independent, piecewise smooth, test function in $Q^+_{\rho}(\theta)$ that vanishes on $\pl K_\rho$.
If we take the level $k\le\mu^-+M$, the spatial integral at $t=0$ 
on the right-hand side of the energy estimate (i.e. the term at the time level $s-S$ in Proposition~\ref{Prop:2:1}) vanishes due to
the assumption that $u(\cdot, 0)\ge \mu^- + M$ a.e. in $K_{\rho}$.
The term involving $\z_t$ also vanishes since $\z$ is independent of $t$.
 As a result, the energy estimate yields for $\tilde{k}<k$ that
\begin{align*}
	&\essup_{0<t<\theta\varrho^p}
	\int_{K_\varrho}\z^p\big(|u|+|k|\big)^{q-1}(u-\tilde{k})_-^2\,\dx
	+
	\iint_{Q^+_\varrho(\theta)}\z^p|D(u-\tilde{k})_-|^p\,\dx\dt\\
	&\qquad\le
	\gm\iint_{Q^+_\varrho(\theta)}(u-k)^{p}_-|D\z|^p\,\dx\dt.
\end{align*}
Introduce $k_n$, $\tilde{k}_n$, $\rho_n$, $\tilde{\rho}_n$, $K_n$ and $\widetilde{K}_n$
as in Lemma~\ref{DG:p-q}. The only difference is that the cylinders $Q_n$ and $\widetilde{Q}_n$ are now of forward type, i.e.
$Q_n=K_n\times(0,\theta\rho^p)$ and $\widetilde{Q}_n=\widetilde{K}_n\times(0,\theta\rho^p)$. Note that while shrinking
the base cubes $K_n$ and $\widetilde{K}_n$ along $\rho_n$, we keep the height of the cylinders fixed.
For the piecewise smooth function $\z(x)$,
we may choose it to vanish on $\pl K_n$, be equal to $1$ in $\widetilde{K}_n$, and satisfy $|D\z|\le\gm 2^n/\rho$.
As a result, we may obtain that
\begin{align*}
	&\essup_{0<t<\theta\varrho^p}
	\int_{\widetilde{K}_n} \big(|u|+|k_n|\big)^{q-1}(u-\tilde{k}_n)_-^2\,\dx
	+
	\iint_{\widetilde{Q}_n}|D(u-\tilde{k}_n)_-|^p \,\dx\dt
	\le
	\gm \frac{2^{pn}}{\varrho^p}M^{p}|A_n|,
\end{align*}
where $\gm$ depends on the data and $A_n=[u< k_n]\cap Q_n$. We may proceed as in Lemma~\ref{DG:p-q}
and obtain the recursive inequality
\begin{equation*}
\begin{aligned}
	 Y_{n+1}
	\le
	&\frac{\gm b^n}{(1-a)^{\frac{N+3}{N+2}}} 
	\left(\frac{\theta}{M^{2-p}\min\{L^{q-1},M^{q-1}\}}\right)^{\frac1{N+2}} Y_n^{1+\frac{1}{N+2}},
\end{aligned}
\end{equation*}
where $Y_n=|A_n|/|Q_n|$, $L=\max\{| \mu^-|, |\mu^-+M|\}$, and the positive constants $\gm$ and 
$b$ depend only on the data. 
Hence, by \cite[Chapter I, Lemma~4.1]{DB}, 
$Y_n\to0$ if we require
$Y_o\le\nu$ where
\begin{align*}
\nu&=\gm^{-1}(1-a)^{N+3}
\frac{M^{2-p}\min\{L^{q-1},M^{q-1}\}}{\theta}
\end{align*}
This requirement, i.e. $Y_o\le\nu$, is fulfilled if $\nu=1$, that is, we take
\[
\theta=\gm^{-1}(1-a)^{N+3}M^{2-p}\min\{L^{q-1},M^{q-1}\}.
\]
We may conclude with such a choice of $\theta$.
\end{proof}

\noi Now we are ready to present\\

\noi{\it \underline{Proof of Theorem~\ref{Thm:pointwise}.}}
Assume $(x,t)=(0,0)$ for simplicity. Recall that we need to show the reverse inequality of \eqref{Eq:3:1}.
Suppose to the contrary that 
\begin{equation}\label{Eq:contrary}
u_*(0,0)=\essliminf_{(y,s)\to(0,0)}u(y,s)<\inf_{\theta>0}\lim_{\rho\to0}\essinf_{Q'_{\rho}(\theta)}u.
\end{equation}
By the definition of $u_*(0,0)$, for any $\varep>0$, there exists $\rho_o(\varep)>0$, such that for all $\rho\in(0,\rho_o)$
\[
\mu^-:=u_*(0,0)-\varep\le\essinf_{Q_\rho}u.
\]
According to \eqref{Eq:contrary}, there exists $M>0$ (independent of $\varep$ and $\theta$), such that for any $\theta>0$ there holds
\[
\lim_{\rho\to0}\essinf_{Q'_{\rho}(\theta)}u\ge\mu^-+M.
\]
This means there exists $\rho_1(M,\theta)>0$, such that for all $\rho\in(0,\rho_1)$
\begin{equation}\label{Eq:lower-half}
\essinf_{Q'_{\rho}(\theta)}u\ge\mu^-+\tfrac12M.
\end{equation}
In particular, we may assume $t_o:=-\theta\rho^p$ is a Lebesgue instant without loss of generality and by \eqref{Eq:lower-half}
\[
u(\cdot, t_o)\ge\mu^-+\tfrac12M\quad\text{ a.e. in }K_\rho.
\]
Now we may use Lemma~\ref{DG:p-q-bdry} (with $a=\frac12$) to determine $\tilde\theta>0$ 
depending only on $\mu^-$, $M$ and the data,
such that
\begin{equation}\label{Eq:upper-half}
u\ge\mu^-+\tfrac14M\quad\text{ a.e. in }(0,t_o)+Q^+_{\frac34\rho}(\tilde\theta).
\end{equation}
To proceed, we choose $\theta$ to satisfy
$t_o+\tilde\theta(\frac34\rho)^p>\tilde\theta(\frac12\rho)^p$, i.e. $\theta<[(\frac34)^p-(\frac12)^p]\tilde\theta$.
Combining \eqref{Eq:upper-half} with \eqref{Eq:lower-half}
and taking $\varep\le\frac18M$,
we arrive at
\[
u\ge\mu^-+\tfrac14M=u_*(0,0)-\varep+\tfrac14M\ge u_*(0,0)+\tfrac18M\quad\text{ a.e. in }Q_{\frac12\rho}(\theta),
\]
for all $\rho\le\min\{\rho_o,\rho_1\}$.
This yields a contradiction to the definition of $u_*(0,0)$ and hence completes the proof.
\hfill $\square$
\begin{remark}\label{Rmk:3:3}\upshape
In the above proof, the final choice of $\theta$ via $\theta<[(\frac34)^p-(\frac12)^p]\tilde\theta$,
is possible only if $M$ is independent of $\theta$, as $\tilde\theta$ generally depends on $M$ according to Lemma~\ref{DG:p-q-bdry}.
However, if $p=2$ and $q=1$, the number $\tilde\theta$ can be selected in terms of the data only.
In such a case, a closer inspection of the proof indicates that $\inf_{\theta>0}$ could be attained. 
\end{remark}
\subsection{A mean value property}\label{S:3:1}
In this section, we use the limit of  certain integral averages to recover 
the point value of $u_*(x,t)$ for every $(x,t)$ in the domain $E_T$.
Since to our knowledge this result has not been written in the literature even for the non-degenerate case,
we will not pursue it for the general doubly nonlinear equation \eqref{Eq:p-q} -- \eqref{Eq:1:2p}. Instead,
concentration will be made on the parabolic $p$-Laplace type equation ($q=1$) and 
the porous medium type equation ($p=2$). 

\begin{proposition}\label{Lm:Lebesgue-pt}
Let $u$ be a locally bounded, local, weak supersolution to \eqref{Eq:p-q} -- \eqref{Eq:1:2p} in $E_T$,
and $u_*$ be the lower semicontinuous representative of $u$ obtained in Theorem~\ref{Thm:lsu-parabolic}.
Suppose that either $q=1$ or $p=2$ holds.
For every $(x,t)\in E_T$, there holds 
\[
\inf_{\theta>0}\lim_{\rho\to0}\bint_{t-2\theta\rho^p}^{t-\theta\rho^p}\bint_{K_{\rho}(x)}|u_*(x,t)-u(y,s)|\,\dy\d s=0.
\]
\end{proposition}
\begin{remark}\upshape
The infimum of $\theta$ can be attained at some real number, if $p=2$ and $q=1$;
see Section~\ref{S:remove-bd}.
\end{remark}
%
\begin{remark}\upshape
Proposition~\ref{Lm:Lebesgue-pt} offers another possible way of proving Theorem~\ref{Thm:pointwise}.
To wit, assume $(x,t)=(0,0)$ for simplicity and observe that for any $\theta>0$
\[
\bint_{-2\theta\rho^p}^{-\theta\rho^p}\bint_{K_{\rho}}|u_*(0,0)-u(y,s)|\,\dy\d s\ge\essinf_{Q'_\rho(\theta)}(u-u_*(0,0)).
\]
This is a refined, parabolic version of the estimate in Remark~\ref{Rmk:2:1}.
Sending $\rho\to0$ and taking $\inf_{\theta>0}$ in the above estimate, Proposition~\ref{Lm:Lebesgue-pt} implies that
\[
u_*(0,0)\ge \inf_{\theta>0}\lim_{\rho\to0}\essinf_{Q'_\rho(\theta)} u.
\]
Hence we obtain the reverse inequality of \eqref{Eq:3:1}.
Nevertheless, the additional assumption in Proposition~\ref{Lm:Lebesgue-pt} is a drawback.
\end{remark}
%

In what follows we will deal with the proof of Proposition~\ref{Lm:Lebesgue-pt}
for the parabolic $p$-Laplace type equation (Section~\ref{S:3:2}) and for the porous medium equation (Section~\ref{S:3:3}) separately.
%
\subsection{The parabolic $p$-Laplace type equation}\label{S:3:2}
Let $u$ be locally bounded in $E_T$. For a compact set $K\subset\rr^N$ and
 a cylinder $\boldsymbol Q:=K\times(T_1,T_2)\subset E_T$
we introduce the numbers $\mu^{\pm}$ satisfying
\[
\mu^+\ge\essup_{\boldsymbol Q}u,\quad \mu^-\le\essinf_{\boldsymbol Q} u.
\]
We also assume $(y,s)\in \boldsymbol Q$, such that the following forward
cylinders
\begin{equation*}
\left\{
\begin{array}{ll}
{\dsty K_{4\rho}(y)\times\left(s,s+\kappa M^{2-p}\rho^p\right)}\quad&\text{ for }1<p<2,\\[5pt]
{\dsty K_{4\rho}(y)\times\left(s,s+\frac{b^{p-2}}{(\eta M)^{p-2}}\kappa\rho^p\right)}\quad&\text{ for }p>2,\\
\end{array}\right.
\end{equation*}
are included in $\boldsymbol Q$,
where $M$ is a positive number and
 the parameters $b$, $\eta$ and $\kappa$ will be determined in the following proposition
in terms of the data only. This can always be made true if we choose $\rho$
small enough. The expansion of positivity can be retrieved from Proposition~4.1 and 
Proposition~5.1 in Chapter 4 of \cite{DBGV-mono}, which has been adapted in the following form in \cite{Liao-JDE}
to deal with the H\"older regularity of weak solutions to the parabolic $p$-Laplace type equation.

\begin{proposition}\label{exp-positivity-1}
Let $u$ be a locally bounded below, local, weak supersolution to \eqref{Eq:p-q} -- \eqref{Eq:1:2p}
with $q=1$ in $E_T$.
 Suppose that $s$ is a Lebesgue instant and
\[
\left|\big[ u(\cdot, s)-\mu^-\ge M\big]\cap K_\rho(y)\right|\ge\al |K_\rho|
\]
for some $M>0$ and $\al\in(0,1)$. When $1<p<2$, 
there exist constants $\kappa,\,\eta\in(0,1)$
depending only on the data and $\al$, such that
\[
u(\cdot, t)-\mu^-\ge \eta M\quad\text{ a.e. in }\ K_{\rho}(y)
\]
for all times
\[
s+\tfrac12\kappa M^{2-p}\rho^p\le t\le s+\kappa M^{2-p}\rho^p.
\]
When $p>2$, there exist constants $b>1$, $\kappa,\,\eta\in(0,1)$,
depending only on the data and $\al$, such that
\[
u(\cdot, t)-\mu^-\ge \eta M\quad\text{ a.e. in }\ K_{\rho}(y)
\]
for all times
\[
s+\frac{b^{p-2}}{(\eta M)^{p-2}}\tfrac12\kappa\rho^p\le t\le s+
\frac{b^{p-2}}{(\eta M)^{p-2}}\kappa\rho^p.
\]
\end{proposition}
\begin{remark}\upshape
The various constants in Proposition~\ref{exp-positivity-1} can be made stable as $p\to2$; 
see \cite[Chapter~4, Section~6]{DBGV-mono}.
Thus it recovers the classical theory (cf. \cite{LSU}).
\end{remark}
\subsubsection{Proof of Proposition~\ref{Lm:Lebesgue-pt} by Proposition~\ref{exp-positivity-1}}\label{S:proof-p}
Suppose to the contrary that for some $(x_o,t_o)\in E_T$
\[
\inf_{\theta>0}\lim_{\rho\to0}\bint_{t_o-2\theta\rho^p}^{t_o-\theta\rho^p}\bint_{K_{\rho}(x_o)}|u_*(x_o,t_o)-u(y,s)|\,\dy\d s=:\dl>0
\]
 For ease of notation, we may assume $(x_o,t_o)=(0,0)$.
Then
there exists 
a sequence of positive numbers $\{\rho_n: n=1,2,\dots\}$ converging to zero,
such that for any $n=1,2,\dots$
\[
\bint_{-2\theta\rho_n^p}^{-\theta\rho_n^p}\bint_{K_{\rho_n}}|u_*(0,0)-u(y,s)|\,\dy\d s\ge\dl\quad
\text{ for an arbitrary but fixed }\theta>0.
\] 

Assuming $\theta$ has been determined for the moment, by the definition of $u_*$,
for any $\varep\in(0,1)$, there exists $\rho_o(\varep,\theta)>0$, such that for any $\rho\in(0,\rho_o)$,
there holds
\begin{equation*}
u_*(0,0)-\varep\le\essinf_{Q_{4\rho}(2\theta)}u.
\end{equation*}
Accordingly we fix $n$ so large that $\rho_n<\rho_o$
and introduce a non-negative function in $Q_{4\rho_n}(2\theta)$ by
\[
v:=u-u_*(0,0)+\varep\ge0.
\]
After such an $n$ is fixed, for ease of notation, we still denote $\rho_n$ by $\rho$.
Setting $Q'_\rho(\theta)=K_{\rho}\times(-2\theta\rho^p,-\theta\rho^p)$ for convenience,
we first estimate the integral average of $v$ over $Q'_\rho(\theta)$ from below with the triangle inequality
\begin{equation}\label{Eq:estimate-above}
\biint_{Q'_\rho(\theta)}v\,\dy\d s\ge
\biint_{Q'_\rho(\theta)}|u_*(0,0)-u(y,s)|\,\dy\d s-\varep\ge\tfrac12\dl,
\end{equation}
provided we choose $\varep\le\frac12\dl$.
On the other hand,  letting $M>0$ to be determined,
 we estimate the integral average of $v$ over $Q'_\rho(\theta)$ from above by
\begin{equation}\label{Eq:estimate-below}
\begin{aligned}
&\biint_{Q'_\rho(\theta)}v\,\dy\d s\\
&\quad\le
\frac1{|Q'_\rho(\theta)|}\iint_{Q'_\rho(\theta)\cap[v\le M]}v\,\dy\d s+
\frac1{|Q'_\rho(\theta)|}\iint_{Q'_\rho(\theta)\cap[v> M]}v\,\dy\d s\\
&\quad\le M+2\mu^+\frac{|[v>M]\cap Q'_\rho(\theta)|}{|Q'_\rho(\theta)|},
\end{aligned}
\end{equation}
where
\[
\mu^+:=1+\essup_{E_T}|u|.
\]
We choose $M=\frac14\dl$, and combine \eqref{Eq:estimate-above} and \eqref{Eq:estimate-below}
to obtain
\[
\frac{|[v>M]\cap Q'_\rho(\theta)|}{|Q'_\rho(\theta)|}\ge\frac{\dl}{4\mu^+}.
\] 
This implies that there exists a Lebesgue instant $s\in [-2\theta\rho^p, -\theta\rho^p]$, such that
\begin{equation}\label{Eq:meas-info}
|[v(\cdot, s)>M]\cap K_{\rho}|\ge\frac{\dl}{4\mu^+}|K_{\rho}|.
\end{equation}
Next we may apply Proposition~\ref{exp-positivity-1} for $\mu^-=u_*(0,0)-\varep$ 
within $\boldsymbol Q=Q_{4\rho}(2\theta)$:
there exist $\gm_1>0$ and $\eta\in(0,1)$ depending on the data and $\dl/\mu^+$, such that
\[
v\ge \eta M\quad\text{ a.e. in } K_{\rho}\times\big(s+\tfrac12\gm_1 M^{2-p}\rho^p,s+\tfrac52\gm_1 M^{2-p}\rho^p\big).
\]
Therefore, noticing that $\dl$ is independent of $\theta$, 
we may choose $\theta=\gm_1 M^{2-p}=\gm_1(\frac14\dl)^{2-p}$, such that the above line implies
\[
v\ge \eta M\quad\text{ a.e. in }  Q_{\rho}(\tfrac12\theta),
\]
which in turn yields
\[
u\ge u_*(0,0)+\eta M-\varep\ge u_*(0,0)+\tfrac12\eta M\quad\text{ a.e. in } Q_{\rho}(\tfrac12\theta),
\]
provided we further restrict the choice of $\varep$ by $\varep\le\frac12\eta M$.
The above line however contradicts the definition of $u_*(0,0)$.
\subsubsection{Boundedness can be removed when $p\ge2$}\label{S:remove-bd}
We provide an amelioration of Proposition~\ref{Lm:Lebesgue-pt} when $p\ge2$. Namely, 
the assumption on the local boundedness of supersolutions could be removed.
Moreover, the infimum of $\theta$ in Proposition~\ref{Lm:Lebesgue-pt} can be attained when $p=2$.
This can be achieved using the weak Harnack inequality established in \cite{Kuusi-08}; 
see also \cite[Chapter~5, Section~7]{DBGV-mono}.
\begin{theorem}\label{Thm:weak-Harnack}
Let $u$ be a non-negative, local, weak supersolution 
to  \eqref{Eq:p-q} -- \eqref{Eq:1:2p}
with $q=1$ and $p>2$.  There exist positive 
constants $c$ and $\gm_o$, depending only on the data, such that for every Lebesgue instant $s\in(0,T)$
\begin{equation*}
\bint_{K_\rho(y)}u(x,s)\,\dx\le c
\Big(\frac{\rho^p}{T-s}\Big)^{\frac1{p-2}}
+\gm_o\essinf_{K_{4\rho}(y)}u(\cdot,t)
\end{equation*}
for all times 
\begin{equation*}
s+{\txty\frac12}\theta\rho^p\le t\le s+\theta\rho^p
\end{equation*}
where 
\begin{equation*}
\theta=\min\left\{c^{2-p}\frac{T-s}{\rho^p}\,,\,\Big[
\bint_{K_\rho(y)}u(x,s)\,\dx\Big]^{2-p}\right\}.
\end{equation*}
\end{theorem}
\begin{remark}\label{Rmk:weak-Harnack}\upshape
{\normalfont If $s$ and $\rho$ are chosen such that
\[
s+\frac {2c^{p-2}}{\Big[\dsty\bint_{K_\rho(y)}u(x,s)\,\dx\Big]^{p-2}}\,\rho^p<T,
\]
then
\begin{equation*}
\begin{array}{cc}
\dsty c\left(\frac{\rho^p}{T-s}\right)^{\frac1{p-2}}<\frac1{2^{\frac1{p-2}}}\bint_{K_\rho(y)}u(x,s)\,\dx,\\[10pt]
\dsty \theta=\Big[\bint_{K_\rho(y)}u(x,s)\,\dx\Big]^{2-p},
\end{array}
\end{equation*}
and therefore, 
\begin{equation}\label{WHI}
\bint_{K_\rho(y)}u(x,s)\,\dx\le\bar\gm\essinf_{K_{4\rho}(y)}u(\cdot,t)
\end{equation}
for all times 
\begin{equation*}
s+{\txty\frac12}\theta\rho^p\le t\le s+\theta\rho^p.
\end{equation*}
Moreover, $\bar\gm=\gm_o(1-2^{\frac1{2-p}})^{-1}$, and therefore the constant is stable as $p\to2$.}
\end{remark}
\noi Another result we will rely on is the following version of Lemma~\ref{DG:p-q-bdry}
for the parabolic $p$-Laplace equation. The proof can be easily extracted from
the proof of Lemma~\ref{DG:p-q-bdry} by letting $q=1$.
\begin{lemma}\label{DG:p-bdry}
Let $u$ be a  local, weak supersolution to the parabolic $p$-Laplace equation in $E_T$.
Assume that $u$ is locally, essentially bounded below in $E_T$.
If  $s$ is a Lebesgue instant and
\[
u(\cdot, s)\ge \mu^- + M\quad\text{ a.e. in }K_\rho(y),
\]
then 
\[
u\ge \mu^- + a M \quad\text{ a.e. in } (y,s)+Q^+_{\frac34 \varrho}(\theta),
\]
where for some $c\in(0,1)$ depending only on the data
\[
\theta=c(1-a)^{N+3}M^{2-p}.
\]
\end{lemma}
\begin{remark}\upshape
For the same purpose, Lemma~\ref{DG:p-bdry} may be replaced by Proposition~\ref{exp-positivity-1} with $\al=1$.
The difference is that  Proposition~\ref{exp-positivity-1} translates measure information into uniform estimates of later times,
whereas Lemma~\ref{DG:p-bdry} propagates quantitative information
without a time lag.
The virtual advantage of Lemma~\ref{DG:p-bdry} is that its proof is much simpler,
and the constant $\gm$ is easily verified to be stable as $p\to2$.
\end{remark}
Now we  prove Proposition~\ref{Lm:Lebesgue-pt} without the boundedness assumption on $u$.
The proof starts just like in Section~\ref{S:proof-p}.
Adopting the notations from Section~\ref{S:proof-p}, the argument departs from \eqref{Eq:estimate-above}.
The lower bound in \eqref{Eq:estimate-above} implies that there exists
a Lebesgue instant $s\in [-2\theta\rho^p, -\theta\rho^p]$, such that
\[
\bint_{K_\rho}v(\cdot, s)\,\dy\ge\tfrac12\dl.
\]
Now by Theorem~\ref{Thm:weak-Harnack} (see in particular \eqref{WHI} in Remark~\ref{Rmk:weak-Harnack}), this implies that
for some $\bar\gm>0$ depending only on the data
\begin{equation}\label{Eq:initial}
v(\cdot,s+\tau)\ge\frac{\dl}{\bar\gm}\quad\text{ a.e. in }K_\rho
\end{equation}
for some $\tau$ satisfying
\[
0\le \tau=\Big[\bint_{K_\rho(y)}v(x,s)\,\dx\Big]^{2-p}\rho^p\le(\tfrac12\dl)^{2-p}\rho^p.
\]
With no loss of generality, let us assume that $s+\tau$ is a Lebesgue instant.
Therefore, noticing that $\dl$ is independent of $\theta$, 
we may choose $\theta=2(\tfrac12\dl)^{2-p}$ and apply Lemma~\ref{DG:p-bdry}
with the initial datum \eqref{Eq:initial} to conclude that there exists $c\in(0,1)$ depending only on the data, such that
\[
v(\cdot, t)\ge\frac{\dl}{2\bar\gm}\quad\text{ a.e. in }K_{\frac12\rho}
\]
for all times
\[
s+\tau<t<s+\tau+c\Big(\frac{\dl}{\bar\gm}\Big)^{2-p}\rho^p.
\]
We claim that the above time interval covers $(-\frac12\theta\rho^p,\frac12\theta\rho^p)$.
Indeed, notice first that $s+\tau\in(-2\theta\rho^p,-\frac12\theta\rho^p)$.
Next when $p>2$, we may choose $\bar\gm$ even smaller to ensure that
$c(\dl/\bar\gm)^{2-p}>3\theta=6(\tfrac12\dl)^{2-p}$.

As a result, there exists $\bar\gm>0$ depending only on the data, such that
\[
v\ge \frac{\dl}{2\bar\gm} \quad\text{ a.e. in }Q_{\frac12\rho}(\theta),
\]
which yields a contradiction just like in Section~\ref{S:proof-p}.

When $p=2$, we may apply Lemma~\ref{DG:p-bdry} finite times (at most $5/c$) to reach a similar contradiction.
Moreover, now the choice of $\theta$ can be made independent of $\dl$.
As a result, the infimum of $\theta$ can actually be attained at a real number.
\hfill $\square$
\begin{remark}\upshape
A weak Harnack inequality for $1<p<2$  remains elusive.
A partial answer based on the comparison principle can be found in \cite[Proposition~3.1]{GLL}.
Whenever it is available, one may perform a similar reasoning as above
to remove the boundedness assumption in Proposition~\ref{Lm:Lebesgue-pt}.
\end{remark}
\subsection{The porous medium type equation}\label{S:3:3}
In this section we deal with the case $p=2$ and $q>0$. 
We shall write \eqref{Eq:p-q} -- \eqref{Eq:1:2p} in a more conventional form of
the porous  medium type equation (cf. Remark~\ref{Rmk:p-q}). 
That is, we shall consider the quasilinear, parabolic partial differential equation
\begin{equation}  \label{Eq:PME}
	u_t-\dvg\bl{A}\big(x,t,u, D(|u|^{m-1}u)\big) = 0\quad \mbox{ weakly in $ E_T$}
\end{equation}
where $m>0$ and the function ${\bf A}(x,t,u,\xi)$ satisfies the structure conditions \eqref{Eq:1:2p} with $p=2$.
For a compact set $K\subset\rr^N$ and
 a cylinder $\boldsymbol Q:=K\times(T_1,T_2)\subset E_T$
we introduce the numbers $\mu^{\pm}$ satisfying
\[
\mu^+\ge\essup_{\boldsymbol Q}u,\quad \mu^-\le\essinf_{\boldsymbol Q} u.
\]
We also assume $(y,s)\in\boldsymbol Q$, such that the following forward
cylinders
\begin{equation*}
\left\{
\begin{array}{ll}
{\dsty K_{4\rho}(y)\times\left(s,s+\kappa M^{1-m}\rho^2\right)}\quad&\text{ for }0<m<1,\\[5pt]
{\dsty K_{4\rho}(y)\times\left(s,s+\frac{b^{m-1}}{(\eta M)^{m-1}}\kappa\rho^2\right)}\quad&\text{ for }m>1,
\end{array}\right.
\end{equation*}
are included in $\boldsymbol Q$,
where $M$ is a positive number and
 the parameters $b$, $\eta$ and $\kappa$ will be determined in the following proposition
in terms of the data only. This can always be made true if we choose $\rho$
small enough. The main tool is the expansion of positivity from Proposition~7.1 and 
Proposition~7.2 in Chapter 4 of \cite{DBGV-mono}, which has been adapted in the following form in \cite{Liao-JDE}
to deal with the H\"older regularity of weak solutions to the porous medium type equation.
The novelty is that it can deal with signed solutions.
\begin{proposition}\label{exp-positivity-2}
Let $u$ be a locally bounded below, local, weak supersolution to the porous medium type 
equation \eqref{Eq:PME} in $E_T$. 
Suppose that $s$ is a Lebesgue instant and
\[
\left|\big[u(\cdot, s)-\mu^-\ge M\big]\cap K_\rho(y)\right|\ge\al |K_\rho|
\]
for some $M>0$ and $\al\in(0,1)$. When $0<m<1$, 
there exist constants $\xi,\,\kappa,\,\eta\in(0,1)$
depending only on the data and $\al$, such that
either $|\mu^{-}|>\xi M$ or
\[
u(\cdot, s)-\mu^-\ge \eta M\quad\text{ a.e. in }\ K_{\rho}(y)
\]
for all times
\[
s+\tfrac12\kappa M^{1-m}\rho^2\le t\le s+\kappa M^{1-m}\rho^2.
\]
When $m>1$, there exist constants $b>1$, $\kappa,\,\eta,\,\xi\in(0,1)$,
depending only on the data and $\al$, such that either $|\mu^{-}|>\xi M$ or
\[
u(\cdot, s)-\mu^-\ge \eta M\quad\text{ a.e. in }\ K_{\rho}(y)
\]
for all times
\[
s+\frac{b^{m-1}}{(\eta M)^{m-1}}\tfrac12\kappa\rho^2\le t\le s+
\frac{b^{m-1}}{(\eta M)^{m-1}}\kappa\rho^2.
\]
\end{proposition}
\subsubsection{Proof of Proposition~\ref{Lm:Lebesgue-pt} by Proposition~\ref{exp-positivity-2}}
We proceed exactly as in Section~\ref{S:proof-p}. Under current setting, we take $p=2$.
After a similar reasoning, we obtain an analog of \eqref{Eq:meas-info}, 
i.e., there exists  a Lebesgue instant $s\in [-2\theta\rho^2, -\theta\rho^2]$, such that
\begin{equation}\label{Eq:meas-info-1}
\left|\big[v(\cdot, s)>M\big]\cap K_{\rho}\right|\ge\frac{\dl}{4\mu^+}|K_{\rho}|
\end{equation}
 for $M=\frac14\dl>0$,
where $\dl$ is independent of $\theta$.
Recall also that $v:=u-\mu^-=u-u_*(0,0)+\varep$ is a non-negative function in $Q_{4\rho}(2\theta)$
and $\mu^+:=1+\essup_{E_T}|u|$.

Starting from the measure information in \eqref{Eq:meas-info-1}, 
 we may apply Proposition~\ref{exp-positivity-2} for $\mu^-=u_*(0,0)-\varep$ within $\boldsymbol Q=Q_{4\rho}(2\theta)$.
Consequently, there exist $\gm_1>0$  and $\eta,\,\xi\in(0,1)$ depending on the data and $\dl/\mu^+$, 
such that either $|\mu^-|>\xi M$ or
\[
v\ge \eta M\quad\text{ a.e. in } K_{\rho}\times\big(s+\tfrac12\gm_1 M^{1-m}\rho^2,s+2\gm_1 M^{1-m}\rho^2\big).
\]
Therefore, noticing that $\dl$ is independent of $\theta$, 
we may choose $\theta=\gm_1 M^{1-m}=\gm_1(\frac14\dl)^{1-m}$, such that either $|\mu^-|>\xi M $ or
\[
v\ge \eta M\quad\text{ a.e. in }  Q_{\rho}(\tfrac12\theta),
\]
which implies
\[
u\ge u_*(0,0)+\eta M-\varep\ge u_*(0,0)+\tfrac12\eta M\quad\text{ a.e. in }  Q_{\rho}(\tfrac12\theta),
\]
provided we further restrict the choice of $\varep$ by $\varep\le\frac12\eta M$.
This however contradicts the definition of $u_*(0,0)$.

Next we consider the case $\mu^-=u_*(0,0)-\varep>\xi M$. 
Recall that $\rho=\rho_n$ has been chosen such that
\begin{equation}\label{Eq:lower-bd}
\xi M<u_*(0,0)-\varep\le\essinf_{Q_{4\rho}(2\theta)}u.
\end{equation}
The lower bound in \eqref{Eq:lower-bd} shows that the partial differential equation \eqref{Eq:PME}
is non-degenerate within $Q_{4\rho}(2\theta)$ since
\[
\xi M\le u\le \mu^+\quad\text{ a.e. in }Q_{4\rho}(2\theta).
\]
By the classical parabolic theory in \cite{LSU}, the measure information in \eqref{Eq:meas-info-1}
will yield quantitative expansion of positivity around the center of the cylinder $Q_{\rho}(\theta)$; see also
Proposition~5.1 and Remark~5.1 in \cite{Liao}.
Precisely, there exists $\tilde{\eta}\in(0,1)$ depending on 
$\xi M$, $\mu^+$, $\dl/\mu^+$, $\theta$ and the data, such that
\[
u-\mu^-=u-u_*(0,0)+\varep\ge\tilde{\eta}M\quad\text{ a.e. in } Q_{\frac12\rho}(\theta).
\]
As before, this yields a contradiction to the definition of $u_*(0,0)$
provided we impose $\varep<\frac12\tilde{\eta}M$.

Finally, we treat the case $\mu^-=u_*(0,0)-\varep<-\xi M$. 
This case can also be attributed to the classical theory in \cite{LSU}.
Indeed, we may choose the level $k<\mu^-+\frac12\xi M$
and $(u-k)_-\z^2$ as a test function against the equation \eqref{Eq:PME}.
Notice that on the set $[u<k]$ there holds
\[
-\mu^+<\mu^-\le u<k<\mu^-+\tfrac12\xi M<-\tfrac12\xi M,
\]
that is,
\[
\tfrac12\xi M\le |u|\chi_{[u<k]}\le \mu^+.
\]
The potential degeneracy at $[u=0]$ has been avoided
and $u$ will be a member of the parabolic De Giorgi class in \cite{LSU, Liao}.
Therefore 
the measure information in \eqref{Eq:meas-info-1}
will yield pointwise information around the center of the cylinder $Q_{\rho}(\theta)$.
Precisely, there exists $\hat{\eta}\in(0,1)$ depending on 
$\xi M$, $\mu^+$, $\dl/\mu^+$, $\theta$ and the data, such that
\[
u-\mu^-=u-u_*(0,0)+\varep\ge\hat{\eta}M\quad\text{ a.e. in } Q_{\frac12\rho}(\theta).
\]
As before, this yields  a contradiction to the definition of $u_*(0,0)$
provided we impose $\varep<\frac12\hat{\eta}M$.

As a result, we have shown that 
\[
\inf_{\theta>0}\lim_{\rho\to0}\bint_{-2\theta\rho^p}^{-\theta\rho^p}\bint_{K_{\rho}}|u_*(0,0)-u(y,s)|\,\dy\d s=0.
\]
Due to the transformation made at the beginning of Section~\ref{S:3:3}, 
we still need to show that for $\vp_m(u):=|u|^{m-1}u$,
\[
\inf_{\theta>0}\lim_{\rho\to0}\bint_{-2\theta\rho^p}^{-\theta\rho^p}\bint_{K_{\rho}}\big|\vp_m\big(u_*(0,0)\big)-\vp_m\big(u(y,s)\big)\big|\,\dy\d s=0.
\]
This, however, can be easily deduced from the proven result.
\begin{remark}\upshape
One wonders if the assumption on the boundedness of supersolutions could be removed just like in Section~\ref{S:remove-bd}.
Although a weak Harnack inequality is available for the porous medium type equation 
(cf. \cite[Chapter~5, Section~17]{DBGV-mono} and \cite{WHI-PME}),
it is unclear how to apply it for signed weak supersolutions.
\end{remark}
\begin{remark}\upshape
When $u_*$ has a sign, we are essentially in the non-degenerate case, 
due to the peculiar structure of the porous medium type equation. 
In such a case, the arguments presented here may be refined to show that 
$\inf_{\theta>0}$ in Proposition~\ref{Lm:Lebesgue-pt} can be attained; see \cite[Theorem~2.1]{Ziemer-82} in this connection.
\end{remark}

\bye